\documentclass[a4paper,11pt]{article}
\usepackage{amsmath}
\usepackage{amsfonts}
\usepackage{amssymb}

\providecommand{\U}[1]{\protect \rule{.1in}{.1in}}
\newtheorem{theorem}{Theorem}

\newtheorem{corollary}[theorem]{Corollary}

\newtheorem{example}{Example}

\newtheorem{proposition}[theorem]{Proposition}
\newtheorem{remark}[theorem]{Remark}

\newenvironment{proof}[1][Proof]{\noindent \textbf{#1.} }{\  \rule{0.5em}{0.5em}}
\include {mak}
\input{tcilatex}
\begin{document}

\begin{center}
{\Large The values of the high order Bernoulli polynomials at integers}

{\Large and the r-Stirling numbers}

{\Large \  \  \ }

{\large Miloud Mihoubi\footnote[1]{{\large {\small This research is
supported by the PNR project }}8/U160/3172..} \  \ and\  \  \ Meriem Tiachachat%
\footnotemark[2]}

USTHB, Faculty of Mathematics, RECITS\ Laboratory,

PB 32 El Alia 16111 Algiers, Algeria.

{\large \footnotemark[1]}mmihoubi@usthb.dz \ {\large \footnotemark[1]}%
miloudmihoubi@gmail.com \ {\large \footnotemark[2]}tiachachatmeriem@yahoo.fr

{\large \  \  \ }
\end{center}

\noindent \textbf{Abstract. }In this paper, we exploit the $r$-Stirling
numbers of both kinds in order to give explicit formulae for the values of
the high order Bernoulli numbers and polynomials of both kinds at integers.
We give also some identities linked the $r$-Stirling numbers and binomial
coefficients.

\noindent \textbf{Keywords. }The $r$-Stirling numbers; the high order
Bernoulli polynomials; binomial coefficients.

\noindent \textbf{MSC 2010:} 11B68; 11B73; 11B83.

\section{Introduction}

The study of the higher order Bernoulli polynomials of the both kinds have
extensively used in various branches of mathematics and have extended in
various directions. For details on the higher order Bernoulli polynomials of
the first kind $B_{n}^{\left( \alpha \right) }\left( x\right) $ one can see 
\cite{sri1,sri2,sri3,tod,zha}, and, for the higher order Bernoulli
polynomials of the second kind $b_{n}^{\left( \alpha \right) }\left(
x\right) $ one can see \cite{ade,car,pra,rom}. These polynomials are defined
by the generating function to be%
\begin{eqnarray}
\underset{n\geq 0}{\sum }B_{n}^{\left( \alpha \right) }\left( x\right) \frac{%
t^{n}}{n!} &=&\left( \frac{t}{\exp \left( t\right) -1}\right) ^{\alpha }\exp
\left( xt\right) ,  \label{b1} \\
\underset{n\geq 0}{\sum }b_{n}^{\left( \alpha \right) }\left( x\right) \frac{%
t^{n}}{n!} &=&\left( \frac{t}{\ln \left( 1+t\right) }\right) ^{\alpha
}\left( 1+t\right) ^{x}.  \label{b3}
\end{eqnarray}%
The numbers $B_{n}^{\left( \alpha \right) }:=B_{n}^{\left( \alpha \right)
}\left( 0\right) $ are the high order Bernoulli numbers of the first kind
and $B_{n}:=B_{n}^{\left( 1\right) }\left( 0\right) $ are the Bernoulli
numbers of the second kind. Also, the numbers $b_{n}^{\left( \alpha \right)
}:=b_{n}^{\left( \alpha \right) }\left( 0\right) $ are the high order
Bernoulli numbers of the second kind and $b_{n}:=b_{n}^{\left( 1\right)
}\left( 0\right) $ are the Bernoulli numbers of the second kind. These
numbers and polynomials are connected to the $r$-Stirling numbers of the
first and second kind $\QATOPD[ ] {n}{k}_{r}$ and $\QATOPD \{ \} {n}{k}_{r}$
introduced by Broder \cite{bro,mez2,mih2}. \newline
Recall that the number $\QATOPD[ ] {n}{k}_{r}$ counts the number of
permutations of the set $\left[ n\right] :=\left \{ 1,\ldots ,n\right \} $
into $k$ cycles such that the elements of the set $\left[ r\right] $ are in
different cycles, and, the number $\QATOPD \{ \} {n}{k}_{r}$ counts the
number of partitions of the set $\left[ n\right] $ into $k$ non-empty
subsets such that the elements of the set $\left[ r\right] $ are in
different subsets. These numbers are determined by the generating function
to be%
\begin{eqnarray}
\underset{n\geq k}{\sum }\QATOPD[ ] {n+r}{k+r}_{r}\frac{t^{n}}{n!} &=&\frac{1%
}{k!}\frac{\left( -\ln \left( 1-t\right) \right) ^{k}}{\left( 1-t\right) ^{r}%
},  \label{b2} \\
\underset{n\geq k}{\sum }\QATOPD \{ \} {n+r}{k+r}_{r}\frac{t^{n}}{n!} &=&%
\frac{1}{k!}\left( \exp \left( t\right) -1\right) ^{k}\exp \left( rt\right) .
\label{b4}
\end{eqnarray}%
By combining (\ref{b1}) and (\ref{b4}) we obtain%
\begin{equation}
B_{n}^{\left( -k\right) }\left( r\right) =\binom{n+k}{k}^{-1}\QATOPD \{ \}
{n+r+k}{k+r}_{r},\  \  \  \ r,k\in \mathbb{N},  \label{a5}
\end{equation}%
and by combining (\ref{b3}) and (\ref{b2}) we get%
\begin{equation}
b_{n}^{\left( -k\right) }\left( -r\right) =\left( -1\right) ^{n}\binom{n+k}{k%
}^{-1}\QATOPD[ ] {n+r+k}{k+r}_{r},\  \ r,k\in \mathbb{N}.  \label{a6}
\end{equation}%
In this paper, we give formulas for the values of the high order Bernoulli
polynomials at integers in terms of the $r$-Stirling numbers of both kinds.
In particular, we may prove that the Bernoulli numbers of the first kind
admit the following representations%
\begin{eqnarray*}
b_{n} &=&\left( n+1\right) \binom{2n}{n}^{-1}\underset{j=0}{\overset{n}{\sum 
}}\frac{\left( -1\right) ^{n+j}}{j+1}\binom{2n}{n+j}\QATOPD[ ] {n+j}{j} \\
&=&n\underset{j=0}{\overset{n}{\sum }}\frac{\left( -1\right) ^{j}}{n+j}%
\binom{2n}{n+j}\QATOPD \{ \} {n+j+1}{j+1}
\end{eqnarray*}%
and the Bernoulli numbers of the second kind admit the following
representations%
\begin{eqnarray*}
B_{n} &=&n\underset{j=0}{\overset{n}{\sum }}\frac{\left( -1\right) ^{n+j}}{%
n+j}\binom{2n}{n+j}\QATOPD[ ] {n+j+1}{j+1} \\
&=&\left( n+1\right) \binom{2n}{n}^{-1}\underset{j=0}{\overset{n}{\sum }}%
\frac{\left( -1\right) ^{j}}{j+1}\binom{2n}{n+j}\QATOPD \{ \} {n+j}{j}.
\end{eqnarray*}%
As consequences, we give in the third section some identities linked $r$%
-Stirling numbers and binomial coefficients. The mathematical tools used are
the identities (\ref{a5}), (\ref{a6}) and the Melzak's formula \cite%
{mel,mel1} given by%
\begin{equation}
f_{n}\left( \alpha +x\right) =\alpha \binom{\alpha +p}{p}\underset{j=0}{%
\overset{p}{\sum }}\frac{\left( -1\right) ^{j}}{\alpha +j}\binom{p}{j}%
f_{n}\left( -j+x\right) .,  \label{3}
\end{equation}%
where $f$ is a polynomial of degree $n\leq p,$ $\binom{x}{k}:=\frac{x\left(
x-1\right) \cdots \left( x-k+1\right) }{k!},\ k\geq 1,$ and $\binom{x}{0}%
:=1. $\newline
We use also the notation \newline
$x^{\underline{n}}=x\left( x-1\right) \cdots \left( x-n+1\right) ,$ $n\geq
1, $ $x^{\overline{0}}=1$ \ and \ $x^{\overline{n}}=x\left( x+1\right)
\cdots \left( x+n-1\right) ,$ $n\geq 1,$ $x^{\overline{0}}=1.$

\section{The values of the high order Bernoulli polynomials at integers}

For such applications of (\ref{3}), we consider the Bernoulli polynomials of
both kinds. Indeed, the definitions (\ref{b1}) and (\ref{b3}) show that $%
B_{n}^{\left( \alpha \right) }\left( 0\right) $ and $b_{n}^{\left( \alpha
\right) }\left( 0\right) $ represent (potential) polynomials in $\alpha $ of
degree $\leq n,$ see \cite[Thm. B, p. 141]{com}. So, the polynomials $%
b_{n}^{\left( \alpha \right) }\left( x\right) $ and $B_{n}^{\left( \alpha
\right) }\left( x\right) $ are also polynomials in $\alpha $ of degree $\leq
n.$\ This help to give new formulas for the high order Bernoulli polynomials
in terms of the $r$-Stirling numbers. The following proposition gives
formulas for the values of the high order Bernoulli polynomials of both
kinds at non-positive integers in terms of the $r$-Stirling numbers of the
first kind.

\begin{proposition}
\label{P1}Let $\alpha $ be a real number and $p,$ $q,$ $r,$ $n$ be
non-negative integers with $p\geq n.$ We have%
\begin{eqnarray*}
b_{n}^{\left( \alpha \right) }\left( -r\right) &=&\left( \alpha +q\right) 
\binom{\alpha +p+q}{p}\underset{j=0}{\overset{p}{\sum }}\frac{\left(
-1\right) ^{n+j}}{\alpha +q+j}\binom{p}{j}\frac{\QATOPD[ ] {n+r+j+q}{r+j+q}%
_{r}}{\binom{n+j+q}{n}}, \\
&&\  \  \  \  \\
B_{n}^{\left( \alpha \right) }\left( -r\right) &=&\left( n+1-\alpha
+q\right) \binom{n+1-\alpha +p+q}{p}\underset{j=0}{\overset{p}{\sum }}\frac{%
\left( -1\right) ^{n+j}}{n+1-\alpha +q+j}\binom{p}{j}\frac{\QATOPD[ ] {%
n+r+j+q+1}{r+j+q+1}_{r+1}}{\binom{n+j+q}{n}}.
\end{eqnarray*}
\end{proposition}

\begin{proof}
Setting $x=-q$ and replace $\alpha $ by $\alpha +q$ in (\ref{3}) to get%
\begin{equation}
f\left( \alpha \right) =\left( \alpha +q\right) \binom{\alpha +q+p}{p}%
\underset{j=0}{\overset{p}{\sum }}\left( -1\right) ^{j}\binom{p}{j}\frac{%
f\left( -j-q\right) }{\alpha +q+j}.  \label{4}
\end{equation}%
By setting $f\left( x\right) =b_{n}^{\left( x\right) }\left( -r\right) $ in (%
\ref{4}) we obtain%
\begin{equation*}
b_{n}^{\left( \alpha \right) }\left( -r\right) =\left( \alpha +q\right) 
\binom{\alpha +p+q}{p}\underset{j=0}{\overset{p}{\sum }}\left( -1\right) ^{j}%
\binom{p}{j}\frac{b_{n}^{\left( -j-q\right) }\left( -r\right) }{\alpha +q+j}.
\end{equation*}%
On using (\ref{a6}), the last identity being the first identity of the
proposition.\newline
Upon using the Carlitz's identity $B_{n}^{\left( \alpha \right) }\left(
x\right) =b_{n}^{\left( n+1-\alpha \right) }\left( x-1\right) ,$ see \cite[%
Eqs. (2.11), (2.12)]{car}, the second identity is equivalent to the first
one.
\end{proof}

\noindent For $p=n$ and $q=0$ in Proposition \ref{P1}, we get the following
corollary.

\begin{corollary}
\label{C2}Let $\alpha $ be a real number and $r,$ $n$ be non-negative
integers$.$ We have%
\begin{eqnarray*}
b_{n}^{\left( \alpha \right) }\left( -r\right) &=&\alpha \binom{\alpha +n}{n}%
\binom{2n}{n}^{-1}\underset{j=0}{\overset{n}{\sum }}\frac{\left( -1\right)
^{n+j}}{\alpha +j}\binom{2n}{n+j}\QATOPD[ ] {n+r+j}{r+j}_{r}, \\
&&\  \  \  \  \\
B_{n}^{\left( \alpha \right) }\left( -r\right) &=&\left( n+1-\alpha \right) 
\binom{2n-\alpha +1}{n}\binom{2n}{n}^{-1}\underset{j=0}{\overset{n}{\sum }}%
\frac{\left( -1\right) ^{n+j}}{n+1-\alpha +j}\binom{2n}{n+j}\QATOPD[ ] {%
n+r+j+1}{r+j+1}_{r+1}.
\end{eqnarray*}
\end{corollary}

\noindent In particular, for $\alpha =1\ $the values of the classical
Bernoulli polynomials at non-positive integers are%
\begin{eqnarray*}
b_{n}\left( -r\right) &:&=b_{n}^{\left( 1\right) }\left( -r\right) =\left(
n+1\right) \binom{2n}{n}^{-1}\underset{j=0}{\overset{n}{\sum }}\frac{\left(
-1\right) ^{n+j}}{j+1}\binom{2n}{n+j}\QATOPD[ ] {n+r+j}{r+j}_{r}, \\
&&\  \  \  \  \\
B_{n}\left( -r\right) &:&=B_{n}^{\left( 1\right) }\left( -r\right) =n%
\underset{j=0}{\overset{n}{\sum }}\frac{\left( -1\right) ^{n+j}}{n+j}\binom{%
2n}{n+j}\QATOPD[ ] {n+r+j+1}{r+j+1}_{r+1}.
\end{eqnarray*}%
These representations show that the classical Bernoulli numbers admit the
representations%
\begin{eqnarray*}
b_{n} &=&\left( n+1\right) \binom{2n}{n}^{-1}\underset{j=0}{\overset{n}{\sum 
}}\frac{\left( -1\right) ^{n+j}}{j+1}\binom{2n}{n+j}\QATOPD[ ] {n+j}{j}, \\
&&\  \  \  \  \\
B_{n} &=&n\underset{j=0}{\overset{n}{\sum }}\frac{\left( -1\right) ^{n+j}}{%
n+j}\binom{2n}{n+j}\QATOPD[ ] {n+j+1}{j+1}.
\end{eqnarray*}

\noindent Similarly to Proposition \ref{P1}, the following proposition gives
formulas for the values of the high order Bernoulli polynomials of both
kinds at non-negative integers in terms of the $r$-Stirling numbers of the
second kind.

\begin{proposition}
\label{P2}Let $\alpha $ be a real number and $p,$ $q,$ $r,$ $n$ be
non-negative integers with $p\geq n.$ We have%
\begin{eqnarray*}
B_{n}^{\left( \alpha \right) }\left( r\right) &=&\left( \alpha +q\right) 
\binom{\alpha +q+p}{p}\underset{j=0}{\overset{p}{\sum }}\frac{\left(
-1\right) ^{j}}{\alpha +q+j}\binom{p}{j}\frac{\QATOPD \{ \}
{n+r+q+j}{r+q+j}_{r}}{\binom{n+q+j}{n}}, \\
&&\  \  \  \  \\
b_{n}^{\left( \alpha \right) }\left( r\right) &=&\left( n+1-\alpha +q\right) 
\binom{n+1-\alpha +q+p}{p}\underset{j=0}{\overset{p}{\sum }}\frac{\left(
-1\right) ^{j}}{n+1-\alpha +q+j}\binom{p}{j}\frac{\QATOPD \{ \}
{n+r+q+j+1}{r+q+j+1}_{r+1}}{\binom{n+q+j}{n}}.
\end{eqnarray*}
\end{proposition}

\begin{proof}
By setting $f\left( x\right) =B_{n}^{\left( x\right) }\left( r\right) $ in (%
\ref{4}) we obtain%
\begin{equation*}
B_{n}^{\left( \alpha \right) }\left( r\right) =\left( \alpha +q\right) 
\binom{\alpha +q+p}{p}\underset{j=0}{\overset{p}{\sum }}\left( -1\right) ^{j}%
\binom{p}{j}\frac{B_{n}^{\left( -j-q\right) }\left( r\right) }{\alpha +q+j}.
\end{equation*}%
On using (\ref{a5}), the last identity being the first identity of the
proposition.\newline
Upon using the Carlitz's identity $b_{n}^{\left( \alpha \right) }\left(
x\right) =B_{n}^{\left( n+1-\alpha \right) }\left( x+1\right) ,$ see \cite[%
Eqs. (2.11), (2.12)]{car}, the second identity is equivalent to the first
one.
\end{proof}

\noindent For $p=n$ and $q=0$ in Proposition \ref{P2}, we get the following
corollary.

\begin{corollary}
\label{C3}Let $\alpha $ be a real number and $r,$ $n$ be non-negative
integers$.$ We have%
\begin{eqnarray*}
B_{n}^{\left( \alpha \right) }\left( r\right) &=&\alpha \binom{\alpha +n}{n}%
\binom{2n}{n}^{-1}\underset{j=0}{\overset{n}{\sum }}\frac{\left( -1\right)
^{j}}{\alpha +j}\binom{2n}{n+j}\QATOPD \{ \} {n+r+j}{r+j}_{r}, \\
&&\  \  \  \  \  \\
b_{n}^{\left( \alpha \right) }\left( r\right) &=&\left( n+1-\alpha \right) 
\binom{2n+1-\alpha }{n}\binom{2n}{n}^{-1}\underset{j=0}{\overset{n}{\sum }}%
\frac{\left( -1\right) ^{j}}{n+1-\alpha +j}\binom{2n}{n+j}\QATOPD \{ \}
{n+r+j+1}{r+j+1}_{r+1}.
\end{eqnarray*}
\end{corollary}

\noindent In particular, for $\alpha =1\ $the values of the classical
Bernoulli polynomials at non-negative integers are%
\begin{eqnarray*}
B_{n}\left( r\right) &:&=B_{n}^{\left( 1\right) }\left( r\right) =\left(
n+1\right) \binom{2n}{n}^{-1}\underset{j=0}{\overset{n}{\sum }}\frac{\left(
-1\right) ^{j}}{j+1}\binom{2n}{n+j}\QATOPD \{ \} {n+r+j}{r+j}_{r}, \\
&&\  \  \  \  \\
b_{n}\left( r\right) &:&=b_{n}^{\left( 1\right) }\left( r\right) =n\underset{%
j=0}{\overset{n}{\sum }}\frac{\left( -1\right) ^{j}}{n+j}\binom{2n}{n+j}%
\QATOPD \{ \} {n+r+j+1}{r+j+1}_{r+1}.
\end{eqnarray*}%
These representations show that the classical Bernoulli numbers admit the
representations%
\begin{eqnarray*}
B_{n} &=&\left( n+1\right) \binom{2n}{n}^{-1}\underset{j=0}{\overset{n}{\sum 
}}\frac{\left( -1\right) ^{j}}{j+1}\binom{2n}{n+j}\QATOPD \{ \} {n+j}{j}, \\
&&\  \  \  \  \\
b_{n} &=&n\underset{j=0}{\overset{n}{\sum }}\frac{\left( -1\right) ^{j}}{n+j}%
\binom{2n}{n+j}\QATOPD \{ \} {n+j+1}{j+1}.
\end{eqnarray*}

\noindent Note that the above formula of Bernoulli numbers $B_{n}$ is
exactly the formula given in \cite[Thm. 3.1]{mut}.

\begin{remark}
The Genocchi numbers $\left( G_{n};n\geq 0\right) $ given by $G_{n}=2\left(
1-2^{2n}\right) B_{2n},$ $n\geq 1,$ see \cite[Proposition 2.1]{her}, can be
written via the above two expressions of $B_{n}$ as%
\begin{eqnarray*}
G_{n} &=&4n\left( 1-2^{2n}\right) \underset{j=0}{\overset{2n}{\sum }}\frac{%
\left( -1\right) ^{j}}{2n+j}\binom{4n}{2n+j}\QATOPD[ ] {2n+j+1}{j+1}, \\
&&\  \  \  \\
G_{n} &=&2\left( 2n+1\right) \left( 1-2^{2n}\right) \binom{4n}{2n}^{-1}%
\underset{j=0}{\overset{2n}{\sum }}\frac{\left( -1\right) ^{j}}{j+1}\binom{4n%
}{2n+j}\QATOPD \{ \} {2n+j}{j}.
\end{eqnarray*}%
Furthermore, the link between the Euler and Bernoulli polynomials via the
identity $E_{n-1}\left( 2x\right) =\frac{2}{n}\left( B_{n}\left( 2x\right)
-2^{n}B_{n}\left( x\right) \right) ,$ see \cite[p. 88]{sri4}, shows that the
values of the Euler polynomials at even integers can be written on using the
above expressions of $B_{n}\left( -r\right) $ and $B_{n}\left( r\right) $ as%
\begin{equation*}
E_{n-1}\left( -2r\right) =\frac{2}{n}\left( B_{n}\left( -2r\right)
-2^{n}B_{n}\left( -r\right) \right) \text{ \ and \ }E_{n-1}\left( 2r\right) =%
\frac{2}{n}\left( B_{n}\left( 2r\right) -2^{n}B_{n}\left( r\right) \right) .
\end{equation*}
\end{remark}

\begin{remark}
It is known that $B_{n}=\underset{j=0}{\overset{n}{\sum }}\left( -1\right)
^{j}\frac{j!}{j+1}\QATOPD \{ \} {n}{j}.$ Similarly of the proof of this
identity, we have%
\begin{equation*}
\underset{n\geq 0}{\sum }B_{n}\left( r\right) \frac{t^{n}}{n!}=\frac{\ln
\left( 1+\exp \left( t\right) -1\right) }{\exp \left( t\right) -1}\exp
\left( rt\right) =\underset{j\geq 0}{\sum }\left( -1\right) ^{j}\frac{j!}{j+1%
}\left( \frac{1}{j!}\left( \exp \left( t\right) -1\right) ^{j}\exp \left(
rt\right) \right)
\end{equation*}%
which gives%
\begin{equation*}
B_{n}\left( r\right) =\underset{j=0}{\overset{n}{\sum }}\left( -1\right) ^{j}%
\frac{j!}{j+1}\QATOPD \{ \} {n+r}{j+r}_{r}.
\end{equation*}
\end{remark}

\section{Identities linked $r$-Stirling numbers and binomial coefficients}

\noindent The above Propositions can be used to deduce relations between the 
$r$-Stirling numbers and binomial coefficients as it is shown by the
following two corollaries..

\begin{corollary}
\label{C1}Let $r,$ $n,$ $p,$ $q,$ $k$ be non-negative integers with $p\geq
n. $ We have%
\begin{eqnarray*}
&&\underset{j=0}{\overset{p}{\sum }}\left( -1\right) ^{j}\binom{j+q}{q}%
\binom{n+k+q+j}{k}\binom{n+k+p+q+1}{p-j}\QATOPD[ ] {n+r+j+q}{r+j+q}_{r} \\
&=&\binom{n+k+q}{q}\underset{j=0}{\overset{n}{\sum }}\left( -1\right) ^{n-j}%
\binom{n+k}{j+k}\QATOPD \{ \} {j+k}{k}\left( r-1\right) ^{n-j}
\end{eqnarray*}%
and%
\begin{eqnarray*}
&&\underset{j=0}{\overset{p}{\sum }}\left( -1\right) ^{j}\binom{j+q}{q}%
\binom{n+k+q+j}{k}\binom{n+k+p+q+1}{p-j}\QATOPD \{ \} {n+r+j+q}{r+j+q}_{r} \\
&=&\binom{n+k+q}{q}\underset{j=0}{\overset{n}{\sum }}\left( -1\right) ^{j}%
\binom{n+k}{j+k}\QATOPD[ ] {j+k}{k}\left( r-1\right) ^{\underline{n-j}}.
\end{eqnarray*}
\end{corollary}

\begin{proof}
For $\alpha =n+k+1$ in the first identities of Propositions \ref{P1} and \ref%
{P2}, the left hand side of the first identity of this corollary is equal to 
$\binom{n+k}{k}b_{n}^{\left( n+k+1\right) }\left( -r\right) $ and the left
hand side of the second identity is equal to $\binom{n+k}{k}B_{n}^{\left(
n+k+1\right) }\left( r\right) \ $and use the fact that:%
\begin{eqnarray*}
B_{n}^{\left( n+k+1\right) }\left( x\right) &=&b_{n}^{\left( -k\right)
}\left( x-1\right) \\
&=&\frac{d^{n}}{dt^{n}}\left( \left( \frac{\ln \left( 1+t\right) }{t}\right)
^{k}\left( 1+t\right) ^{x-1}\right) _{t=0} \\
&=&\frac{n!}{\left( n+k\right) !}\frac{d^{n+k}}{dt^{n+k}}\left. \left(
\left( \ln \left( 1+t\right) \right) ^{k}\left( 1+t\right) ^{x-1}\right)
\right \vert _{t=0} \\
&=&\frac{n!}{\left( n+k\right) !}\left. \underset{j=0}{\overset{n+k}{\sum }}%
\binom{n+k}{j}\frac{d^{j}}{dt^{j}}\left( \ln \left( 1+t\right) \right) ^{k}%
\frac{d^{n+k-j}}{dt^{n+k-j}}\left( 1+t\right) ^{x-1}\right \vert _{t=0} \\
&=&\frac{1}{\binom{n+k}{k}}\underset{j=0}{\overset{n}{\sum }}\left(
-1\right) ^{j}\binom{n+k}{j+k}\QATOPD[ ] {j+k}{k}\left( x-1\right) ^{%
\underline{n-j}}
\end{eqnarray*}%
and%
\begin{eqnarray*}
b_{n}^{\left( n+k+1\right) }\left( x\right) &=&B_{n}^{\left( -k\right)
}\left( x+1\right) \\
&=&\frac{d^{n}}{dt^{n}}\left( \left( \frac{\exp \left( t\right) -1}{t}%
\right) ^{k}\exp \left( \left( x+1\right) t\right) \right) _{t=0} \\
&=&\frac{n!}{\left( n+k\right) !}\left. \underset{j=0}{\overset{n+k}{\sum }}%
\binom{n+k}{j}\frac{d^{j}}{dt^{j}}\left( \exp \left( t\right) -1\right) ^{k}%
\frac{d^{n+k-j}}{dt^{n+k-j}}\left( \exp \left( \left( x+1\right) t\right)
\right) \right \vert _{t=0} \\
&=&\frac{1}{\binom{n+k}{k}}\underset{j=0}{\overset{n}{\sum }}\binom{n+k}{j+k}%
\QATOPD \{ \} {j+k}{k}\left( x+1\right) ^{n-j}.
\end{eqnarray*}
\end{proof}

\noindent Below, we present some particular cases of Corollary \ref{C1}.

\begin{example}
For $r=1$ in Corollary \ref{C1} we get%
\begin{eqnarray*}
&&\underset{j=0}{\overset{p}{\sum }}\left( -1\right) ^{j}\binom{j+q}{q}%
\binom{n+k+q+j}{k}\binom{n+k+p+q+1}{p-j}\QATOPD[ ] {j+n+q+1}{j+q+1} \\
&=&\binom{n+k+q}{q}\QATOPD \{ \} {n+k}{k}, \\
&&\  \  \  \  \  \\
&&\underset{j=0}{\overset{p}{\sum }}\left( -1\right) ^{n-j}\binom{j+q}{q}%
\binom{n+k+q+j}{k}\binom{n+k+p+q+1}{p-j}\QATOPD \{ \} {j+n+q+1}{j+q+1} \\
&=&\binom{n+k+q}{q}\QATOPD[ ] {n+k}{k}.
\end{eqnarray*}
\end{example}

\begin{example}
For $k=0$ in Corollary \ref{C1} we get%
\begin{eqnarray*}
\underset{j=0}{\overset{p}{\sum }}\left( -1\right) ^{n-j}\binom{j+q}{q}%
\binom{n+p+q+1}{p-j}\QATOPD[ ] {n+r+j+q}{r+j+q}_{r} &=&\binom{n+q}{q}\left(
r-1\right) ^{n}, \\
&&\  \  \  \  \\
\underset{j=0}{\overset{p}{\sum }}\left( -1\right) ^{j}\binom{j+q}{q}\binom{%
n+p+q+1}{p-j}\QATOPD \{ \} {n+r+j+q}{r+j+q}_{r} &=&\binom{n+q}{q}\left(
r-1\right) ^{\underline{n}}.
\end{eqnarray*}
\end{example}

\begin{example}
For $n=0$ in Corollary \ref{C1} we get%
\begin{equation*}
\underset{j=0}{\overset{p}{\sum }}\left( -1\right) ^{j}\binom{j+q}{q}\binom{%
k+q+j}{k}\binom{k+q+p+1}{p-j}=\binom{k+q}{q}.
\end{equation*}
\end{example}

\begin{corollary}
\label{C5}Let $r,$ $n,$ $p,$ $q,$ $k$ be non-negative integers with $p\geq n.
$ We have%
\begin{eqnarray*}
&&\underset{j=0}{\overset{p}{\sum }}\left( -1\right) ^{j}\binom{n+p+q+k+1}{%
p-j}\binom{q+j}{j}\binom{q+k+1+j}{k}\QATOPD[ ] {n+r+q+k+1+j}{r+q+k+1+j}_{r}
\\
&=&\left( -1\right) ^{n}\binom{n+p+q+k+1}{n+k}\QATOPD[ ] {n+k+r}{k+r}_{r}
\end{eqnarray*}%
and%
\begin{eqnarray*}
&&\underset{j=0}{\overset{p}{\sum }}\left( -1\right) ^{j}\binom{n+p+q+k+1}{%
p-j}\binom{q+j}{j}\binom{q+k+1+j}{k}\QATOPD \{ \} {n+r+q+k+1+j}{r+q+k+1+j}_{r}
\\
&=&\binom{n+p+q+k+1}{n+k}\QATOPD \{ \} {n+k+r}{k+r}_{r}.
\end{eqnarray*}
\end{corollary}

\begin{proof}
Choice $\alpha =-k$ with $k+1\leq q$ in the first identities of Propositions %
\ref{P1} and \ref{P2}, the left hand side of the first identity of this
corollary is equal to $\binom{n+k}{k}\binom{n+p+q}{n+k}b_{n}^{\left(
-k\right) }\left( -r\right) $ and the left hand side of the second identity
is equal to $\binom{n+k}{k}\binom{n+p+q}{n+k}B_{n}^{\left( -k\right) }\left(
r\right) \ $and use the identities given by (\ref{a5}) and (\ref{a6}). After
that, replace $q$ by $q+k+1.$
\end{proof}

\begin{example}
For $k=0$ in Corollary \ref{C5} we get%
\begin{eqnarray*}
\underset{j=0}{\overset{p}{\sum }}\left( -1\right) ^{j}\binom{n+p+q+1}{p-j}%
\binom{q+j}{j}\QATOPD[ ] {n+r+q+1+j}{r+q+1+j}_{r} &=&\binom{n+p+q+1}{n}r^{%
\overline{n}}, \\
&&\  \  \  \  \  \\
\underset{j=0}{\overset{p}{\sum }}\left( -1\right) ^{j}\binom{n+p+q+1}{p-j}%
\binom{q+j}{j}\QATOPD \{ \} {n+r+q+1+j}{r+q+1+j}_{r} &=&\binom{n+p+q+1}{n}%
r^{n}.
\end{eqnarray*}
\end{example}

\begin{example}
For $n=0$ in Corollary \ref{C5} we get%
\begin{equation*}
\underset{j=0}{\overset{p}{\sum }}\left( -1\right) ^{j}\binom{p+q+k+1}{p-j}%
\binom{q+j}{q}\binom{q+k+1+j}{k}=\binom{p+q+k+1}{k}.
\end{equation*}
\end{example}

\end{document}